\documentclass[11pt]{amsart}

\usepackage{mathptmx,amsmath, amssymb, amscd, paralist,tabularx,supertabular,amsmath, verbatim, amsthm, amssymb, mathrsfs, manfnt,latexsym,amscd,graphicx,hyperref,color}

\evensidemargin \oddsidemargin
\author{Benjamin Linowitz}
\address{Department of Mathematics\\Oberlin College\\Oberlin, OH 44074}
\email{benjamin.linowitz@oberlin.edu}

\author{D. B. McReynolds}
\address{Department of Mathematics\\Purdue University\\West Lafayette, IN 47907}
\email{dmcreyno@purdue.edu}

\author{Paul Pollack}
\address{Department of Mathematics\\University of Georgia\\Athens, GA 30602}
\email{pollack@uga.edu}

\author{Lola Thompson}
\address{Department of Mathematics\\Oberlin College\\Oberlin, OH 44074}
\email{lola.thompson@oberlin.edu}

\title{Bounded gaps between primes and the length spectra \\ of arithmetic hyperbolic $3$--orbifolds}
\usepackage{amsmath,amssymb,amsthm,url}
\usepackage{geometry}

\usepackage{fancyhdr}

\fancypagestyle{plain}

\DeclareMathAlphabet{\curly}{U}{rsfs}{m}{n}

\DeclareMathOperator{\Ram}{Ram}

\DeclareMathOperator{\Gal}{Gal}

\DeclareMathOperator{\Mat}{M}

\DeclareMathOperator{\N}{N}

\DeclareMathOperator{\PSL}{PSL}

\DeclareMathOperator{\SL}{SL}

\DeclareMathOperator{\tr}{tr}

\newtheorem{thm}{Theorem}[section]
\newtheorem{cor}[thm]{Corollary}

\newtheorem{prop}[thm]{Proposition}
\newtheorem{lem}[thm]{Lemma}
\theoremstyle{definition}

\theoremstyle{remark}

\newtheorem{proposition}{Proposition}[section]
\def\1{\mathbf{1}}

\theoremstyle{remark}

\theoremstyle{plain}
\newtheorem{theorem}[proposition]{Theorem}

\newtheorem{ques}{Question}

\def\C{\mathbf{C}}
\def\Cc{\curly{C}}
\def\Q{\mathbf{Q}}

\def\Pp{\curly{P}}

\def\N{\mathbf{N}}
\def\Q{\mathbf{Q}}
\def\Z{\mathbf{Z}}

\def\1{\mathbf{1}}

\def\Gal{\mathrm{Gal}}

\newcommand{\abs}[1]{\left\vert#1\right\vert}
\newcommand{\set}[1]{\left\{#1\right\}}

\newcommand{\leg}[2]{\genfrac{(}{)}{}{}{#1}{#2}}
\newcommand{\legs}[2]{\genfrac{[}{]}{}{}{#1}{#2}}

\makeatletter
\def\moverlay{\mathpalette\mov@rlay}
\def\mov@rlay#1#2{\leavevmode\vtop{%
   \baselineskip\z@skip \lineskiplimit-\maxdimen
   \ialign{\hfil$\m@th#1##$\hfil\cr#2\crcr}}}
\newcommand{\charfusion}[3][\mathord]{
    #1{\ifx#1\mathop\vphantom{#2}\fi
        \mathpalette\mov@rlay{#2\cr#3}
      }
    \ifx#1\mathop\expandafter\displaylimits\fi}
\makeatother

\makeatletter
\let\@@pmod\pmod
\DeclareRobustCommand{\pmod}{\@ifstar\@pmods\@@pmod}
\def\@pmods#1{\mkern4mu({\operator@font mod}\mkern 6mu#1)}
\makeatother

\begin{document}

\begin{abstract}
In 1992, Reid asked whether hyperbolic $3$--manifolds with the same geodesic length spectra are necessarily commensurable. While this is known to be true for arithmetic hyperbolic $3$--manifolds, the non-arithmetic case is still open. Building towards a negative answer to this question, Futer and Millichap recently constructed infinitely many pairs of non-commensurable, non-arithmetic hyperbolic $3$--manifolds which have the same volume and whose length spectra begin with the same first $m$ geodesic lengths. In the present paper, we show that this phenomenon is surprisingly common in the arithmetic setting. In particular, given any arithmetic hyperbolic $3$--orbifold derived from a quaternion algebra, any finite subset $S$ of its geodesic length spectrum, and any $k \geq 2$, we produce infinitely many $k$--tuples of arithmetic hyperbolic $3$--orbifolds which are pairwise non-commensurable, have geodesic length spectra containing $S$, and have volumes lying in an interval of (universally) bounded length. The main technical ingredient in our proof is a bounded gaps result for prime ideals in number fields lying in Chebotarev sets which extends recent work of Thorner.
\end{abstract}

\maketitle

\section{Introduction}

Given a closed, negatively curved Riemannian manifold $M$ with fundamental group $\pi_1(M)$, each $\pi_1(M)$--conjugacy class $[\gamma]$ has a unique geodesic representative. The multi-set of lengths of these closed geodesics is called the \textbf{geodesic length spectrum} and is denoted by $\mathcal{L}(M)$. The extent to which $\mathcal{L}(M)$ determines $M$ is a basic problem in geometry and is the main topic of the present paper. Specifically,  our interest lies with the following question, which was posed and studied by Reid \cite{R, R-Survey}:

\begin{ques}\label{Q:Reid}
If $M_1,M_2$ are complete, orientable, finite volume hyperbolic $n$--manifolds and $\mathscr L(M_1)=\mathscr L(M_2)$, then are $M_1,M_2$ commensurable?
\end{ques}

The motivation for this question is two-fold. First, Reid \cite{R} gave an affirmative answer to Question \ref{Q:Reid} when $n=2$ and $M_1$ is arithmetic. In particular, if $M_1$ is arithmetic and $\mathcal{L}(M_1) = \mathcal{L}(M_2)$, then $M_1,M_2$ are commensurable and hence $M_2$ is also arithmetic as arithmeticity is a commensurability invariant. Second, the two most common constructions of Riemannian manifolds with the same geodesic length spectra (Sunada \cite{S}, Vign\'eras \cite{V}) both produce manifolds that are commensurable. Question \ref{Q:Reid} has been extensively studied in the arithmetic setting (i.e.,~when $M_1$ is arithmetic). When $n=3$, Chinburg--Hamilton--Long--Reid \cite{CHLR} gave an affirmative answer. Prasad--Rapinchuk \cite{PR} later showed that the geodesic length spectrum of an arithmetic hyperbolic $n$--manifold determines the manifold up to commensurability when $n\not\equiv 1\pmod{4}$ and $n\neq 7$. Most recently, Garibaldi \cite{G} has confirmed the question in dimension $n=7$.

In the non-arithmetic setting (i.e.,~when neither $M_1$ nor $M_2$ are arithmetic), the relationship between the geodesic length spectrum and commensurability class of the manifold is rather mysterious. To our knowledge, the only explicit work in this area is Millichap \cite{CM} and Futer--Millichap \cite{FM}. In \cite{FM}, which extends work from \cite{CM}, Futer and Millichap produce, for every $m \geq 1$, infinitely many pairs of non-commensurable hyperbolic $3$--manifolds which have the same volume and the same $m$ shortest geodesic lengths. Additionally, they give an upper bound on the volume of their manifolds as a function of $m$. Inspired by \cite{FM}, in this paper we consider the following question.

\begin{ques}
Let $M$ be a hyperbolic $3$--orbifold and $S$ be a finite subset of $\mathscr L(M)$. What can one say about the set of hyperbolic $3$--orbifolds $N$ which are not commensurable with $M$ and for which $\mathscr L(N)$ contains $S$?
\end{ques}

This question was previously studied by the authors in \cite{LMPT}. Let $\pi(V,S)$ denote the maximum cardinality of a collection of pairwise non-commensurable arithmetic hyperbolic $3$--orbifolds derived from quaternion algebras, each of which has volume less than $V$ and geodesic length spectrum containing $S$. In \cite{LMPT}, it was shown that, if $\pi(V,S)\to\infty$ as $V\to\infty$, then there are integers $1\leq r,s\leq |S|$ and constants $c_1,c_2>0$ such that \[\frac{c_1V}{\log(V)^{1-\frac{1}{2^r}}} \leq \pi(V,S) \leq \frac{c_2V}{\log(V)^{1-\frac{1}{2^s}}}\] for all sufficiently large $V$. This shows that not only is it quite common for an arithmetic hyperbolic $3$--orbifold to share large portions of its geodesic length spectrum with other (non-commensurable) arithmetic hyperbolic $3$--orbifolds, but that the cardinality of sets of commensurability classes of such orbifolds grows relatively fast. We note that the hypothesis that $\pi(V,S)\to\infty$ as $V\to\infty$ is necessary as there exist subsets $S$ for which $\pi(V,S)$ is non-zero yet eventually constant. Examples of such sets were given in \cite{L} in the context of hyperbolic surfaces using a construction that easily generalizes to hyperbolic $3$--manifolds.

We now state our main geometric result.

\begin{thm}\label{maintheorem}
Let $M$ be an arithmetic hyperbolic $3$--orbifold which is derived from a quaternion algebra and let $S$ be a finite subset of the length spectrum of $M$. Suppose that $\pi(V,S)\to\infty$ as $V\to\infty$. Then, for every $k\geq 2$, there is a constant $C>0$ such that there are infinitely many $k$--tuples $M_1,\dots, M_k$ of arithmetic hyperbolic $3$--orbifolds which are pairwise non-commensurable, have length spectra containing $S$, and volumes satisfying $\abs{\mathrm{vol}(M_i)-\mathrm{vol}(M_j)}<C$ for all $1\leq i,j\leq k$.
\end{thm}

We note that the main novelty of Theorem \ref{maintheorem} is that we are able to impose a great amount of control on the volumes of the orbifolds $M_1,\dots,M_k$. As a corollary to Theorem \ref{maintheorem} we are able to show (see Corollary \ref{cor:manifolds}) that, when $M$ is a hyperbolic $3$--manifold arising from the elements of reduced norm one in a maximal quaternion order, the orbifolds $M_1,\dots, M_k$ produced by Theorem \ref{maintheorem} may be taken to be manifolds.

The main technical ingredient in the proof of Theorem \ref{maintheorem} is a result showing that there are bounded gaps between prime ideals in number fields which lie in certain Chebotarev sets (see Theorem \ref{thm:algreduction}). This extends a theorem of Thorner \cite{thorner14}. All of these results stem from the seminal work of Zhang \cite{zhang} and Maynard--Tao \cite{maynard} on bounded gaps between primes. The techniques employed by Maynard and Tao, in particular, have proven fruitful in resolving a wide array of interesting questions within number theory. The present paper is yet another example of the impact of their ideas.

\section{Arithmetic hyperbolic $3$--orbifolds}

In this brief section, we review the construction of arithmetic lattices in $\PSL(2,\C)$. For a more detailed treatment of this topic, we refer the reader to \cite{MR}. Given a number field $K$ with ring of integers $\mathcal{O}_K$ and a $K$--quaternion algebra $B$, the set of places of $K$ which ramify in $B$ will be denoted by $\Ram(B)$. It is known that $\Ram(B)$ is a finite set of even cardinality. The subset of $\Ram(B)$ consisting of the finite (resp.~infinite) places of $K$ which ramify in $B$ will be denoted by $\Ram_f(B)$ (resp.~$\Ram_\infty(B)$). By the Albert--Brauer--Hasse--Noether theorem, if $B_1$ and $B_2$ are quaternion algebras over $K$, then $B_1\cong B_2$ if and only if $\Ram(B_1)=\Ram(B_2)$. An \textbf{order} of $B$ is a subring $\mathcal O < B$ which is finitely-generated as an $\mathcal O_K$--module and with $B = \mathcal{O} \otimes_{\mathcal{O}_K} K$. An order is \textbf{maximal} if it is maximal with respect to the partial order induced by inclusion.

Fixing a maximal order $\mathcal O < B$, we will denote by $\mathcal O^1$ the multiplicative group consisting of the units of $\mathcal O$ with reduced norm $1$. Via $B\otimes_K K_\nu \cong \Mat(2,\C)$, the image of $\mathcal O^1$ in $\PSL(2,\C)$ is a discrete subgroup with finite covolume which we will denote by $\Gamma^1_{\mathcal O}$. The group $\Gamma^1_{\mathcal O}$ is cocompact precisely when $B$ is a division algebra. A subgroup $\Gamma$ of $\PSL(2,\C)$ is an \textbf{arithmetic Kleinian group} if it is commensurable with a group of the form $\Gamma^1_{\mathcal O}$. A hyperbolic $3$--orbifold $M={\bf H}^3/\Gamma$ is \textbf{arithmetic} if its orbifold fundamental group $\pi_1(M)=\Gamma$ is an arithmetic Kleinian group. An arithmetic hyperbolic $3$--orbifold is \textbf{derived from a quaternion algebra} if its fundamental group is contained in a group of the form $\Gamma^1_{\mathcal O}$.

For a discrete subgroup $\Gamma < \PSL(2,\C)$, the \textbf{invariant trace field} $K\Gamma$ of $\Gamma$ is the field $\Q(\tr(\gamma^2) : \gamma\in \Gamma)$. Provided $\Gamma$ is a lattice, the invariant trace field is a number field by Weil Local Rigidity. We define $B\Gamma$ to be the $K\Gamma$--subalgebra of $\mathrm{M}(2,\C)$ generated by $\set{\gamma^2~:~\gamma \in \Gamma}$. Provided $\Gamma$ is non-elementary, which is the case when $\Gamma$ is a lattice, $B\Gamma$ is a quaternion algebra over $K\Gamma$ which is called the \textbf{invariant quaternion algebra} of $\Gamma$. The invariant trace field and invariant quaternion algebra of an arithmetic hyperbolic $3$--orbifold are complete commensurability class invariants in the sense that, if $\Gamma_1$ and $\Gamma_2$ are arithmetic Kleinian groups, then the arithmetic hyperbolic $3$--orbifolds ${\bf H}^3/\Gamma_1$ and ${\bf H}^3/\Gamma_2$ are commensurable if and only if $K\Gamma_1\cong K\Gamma_2$ and $B\Gamma_1\cong B\Gamma_2$ (see \cite[Ch 8.4]{MR}).

\section{Bounded gaps between primes in number fields}

For the number-theoretic background assumed in this section, we refer the reader to \cite[Ch 3, \S\S 2 -- 3]{Janusz}. Before stating our bounded gap result, we set some notation. Suppose that $F/K$ is a Galois extension of number fields. By a \textbf{prime ideal of a number field}, we mean a nonzero prime ideal of its ring of integers. Let $P$ be a prime ideal of $K$ unramified in $F$, and let $Q$ be a prime ideal of $F$ lying above $P$. We let $\legs{F/K}{Q} \in \Gal(F/K)$ denote the Frobenius automorphism associated to $Q$. Replacing $Q$ with a different prime $Q'$ above $P$ replaces $\legs{F/K}{Q}$ with $\sigma \legs{F/K}{Q} \sigma^{-1}$ for a certain $\sigma \in \Gal(F/K)$; thus, it makes sense to define the Frobenius conjugacy class $\leg{F/K}{P}$ as the conjugacy class of $\legs{F/K}{Q}$ (inside $\Gal(F/K)$) for an arbitrary prime $Q$ of $F$ lying above $P$.

\begin{theorem}\label{thm:algreduction}
Let $L/K$ be a Galois extension of number fields, let $\Cc$ be a conjugacy class of $\Gal(L/K)$, and let $k$ be a positive integer. Then, for a certain constant $c = c_{L/K,\Cc,k}$, there are infinitely many $k$--tuples $P_1, \dots, P_k$ of prime ideals of $K$ for which the following hold:
\begin{enumerate}
\item $\leg{L/K}{P_1} = \dots = \leg{L/K}{P_k} = \Cc$,
\item $P_1,\dots, P_k$ lie above distinct rational primes,
\item each of $P_1, \dots, P_k$ has degree $1$,
\item $|N(P_i)-N(P_j)| \leq c$, for each pair of $i, j \in \{1,2, \dots, k\}$.
\end{enumerate}
\end{theorem}

When $K=\Q$, Theorem \ref{thm:algreduction} was proved by Thorner \cite{thorner14}. The following proposition allows us to reduce to that case.

\begin{prop}\label{prop:algreduction} Let $L/K$ be a Galois extension of number fields, let $\Cc$ be a conjugacy class of $\Gal(L/K)$, and let $F$ be the Galois closure of $L/\Q$. There is a conjugacy class $\Cc'$ of $\Gal(F/\Q)$ for which the following holds. If $p \in \N$ is a prime for which $\leg{F/\Q}{p} = \Cc'$, then there is a prime ideal $P$ of $K$ lying above $p$ for which
\begin{enumerate}
\item $\leg{L/K}{P}=\Cc$,
\item $N(P)=p$.
\end{enumerate}
\end{prop}

\begin{proof} The Chebotarev density theorem guarantees that a positive proportion of the prime ideals $P$ of $K$ satisfy $\leg{L/K}{P}=\Cc$. Since almost all prime ideals of $K$ have degree $1$ and only finitely many rational primes ramify in $F$, we may fix a prime ideal $P_0$ of $K$ with $\leg{L/K}{P_0}=\Cc$, with $P_0$ having degree $1$, and with $P_0\cap \Z = p_0\Z$ (say) unramified in $F$. Let $Q_0$ be a prime ideal of $F$ lying above $P_0$. We claim that $\Cc' = \leg{F/\Q}{p_0}$ has the desired properties. Indeed, suppose that $p$ is a rational prime with $\leg{F/\Q}{p} = \Cc'$. Since $\leg{F/\Q}{p} = \leg{F/\Q}{p_0}$ and $\leg{F/\Q}{p_0}$ is the conjugacy class of $\legs{F/\Q}{Q_0}$, we may choose a prime ideal $Q$ of $F$ lying above $p$ with $\legs{F/\Q}{Q} = \legs{F/\Q}{Q_0}$. Setting $P = Q_0 \cap \mathcal{O}_K$, we see that $P$ is a prime ideal of $K$ lying above $p$.

We proceed to show that (i) and (ii) hold for this choice of $P$. Note first that, with $f(\cdot/\cdot)$ denoting the inertia degree and $D(\cdot/\cdot)$ denoting the decomposition group,
\begin{equation}\label{Gap:Eq1}
f(P/p) = \frac{f(Q/p)}{f(Q/P)} = \frac{\abs{D(Q/p)}}{\abs{D(Q/P)}} = \frac{\abs{D(Q/p)}}{\abs{(D(Q/p) \cap \Gal(F/K))}}.
\end{equation}
Similarly,
\begin{equation}\label{Gap:Eq2}
f(P_0/p_0) = \frac{\abs{D(Q_0/p_0)}}{\abs{(D(Q_0/p_0) \cap \Gal(F/K))}}.
\end{equation}
Now, $D(Q/p)$ is cyclic and generated by $\legs{F/\Q}{Q}$, while $D(Q_0/p_0)$ is generated by $\legs{F/\Q}{Q_0}$. Since $\legs{F/\Q}{Q} = \legs{F/\Q}{Q_0}$, we have $D(Q/p) = D(Q_0/p_0)$, and so $f(P/p) = f(P_0/p_0)$ via \eqref{Gap:Eq1}, \eqref{Gap:Eq2}. We chose $P_0$ to have degree $1$, and so $f(P/p)=1$. This proves property (ii). To show (i), note that $\leg{L/K}{P}$ is the conjugacy class of $\legs{L/K}{Q\cap L} = \legs{F/K}{Q}\bigg\rvert_{L} = \legs{F/\Q}{Q}\bigg\rvert_{L}$.
The last equality uses that $P$ has degree $1$, so that $\legs{F/K}{Q} = \legs{F/\Q}{Q}$. Similarly, $\leg{L/K}{P_0} = \legs{F/\Q}{Q_0}\big\rvert_{L}$. Since $\legs{F/\Q}{Q}=\legs{F/\Q}{Q_0}$, it follows that $\leg{L/K}{P} = \leg{L/K}{P_0}=\Cc$, which is (i).
\end{proof}

\begin{proof}[Proof of Theorem \ref{thm:algreduction}]
Choose $F$ and $\Cc'$ as in Proposition \ref{prop:algreduction}. By that proposition, it suffices to show that if $\Pp$ is the set of primes $p$ with $\leg{F/\Q}{p} = \Cc'$, then there are infinitely many $k$--tuples of elements of $\Pp$ lying in bounded length intervals. This is a direct consequence of Thorner's generalization of the Maynard--Tao theorem to Chebotarev sets \cite[Thm 1]{thorner14}.
\end{proof}

\section{Proof of Theorem \ref{maintheorem}}

Let $M={\bf H}^3/\Gamma$ be a compact arithmetic hyperbolic $3$--orbifold which is derived from a quaternion algebra $B$ over $K$ and let $S=\{\ell_1,\dots, \ell_r\}$ be a finite subset of the length spectrum of $M$. For each $i=1,\dots r$, let $\gamma_i$ be a loxodromic element of $\Gamma$ whose axis in ${\bf H}^3$ projects to a closed geodesic in $M$ having length $\ell_i$, and let $\lambda_i$ be the eigenvalue of a lift of $\gamma_i$ to $\SL(2,\mathbf C)$ for which $|\lambda_i|>1$. For each $i=1,\dots, r$, we let $L_i=K(\lambda_i)$ and $\Omega_i\subset L_i$ be a quadratic $\mathcal O_K$--order containing a preimage in $L_i$ of $\gamma_i$.

\begin{lem}\label{lenspeclem}
Let $B'$ be a quaternion algebra over $K$ for which $\Ram(B)\subsetneq \Ram(B')$ and $\Ram_f(B)\neq\emptyset$. If $B'$ admits embeddings of $L_1,\dots, L_r$ then the commensurability class defined by $(K,B')$ contains a hyperbolic $3$--orbifold $M'$ which is not commensurable to $M$ and has length spectrum containing $S$. In fact, $M'$ can be taken to be of the form $M'={\bf H}^3/\Gamma_{\mathcal O'}^1$, where $\mathcal O'$ is a maximal order of $B'$.
\end{lem}
\begin{proof}
Let $B'$ be as in the statement of the lemma and $\mathcal O'$ be a maximal order of $B'$. Because $K$ is the invariant trace field and $B$ is the invariant trace field of an arithmetic Kleinian group, the field $K$ is a number field with a unique complex place and the set $\Ram(B)$ contains all real places of $K$. By hypothesis, $\Ram(B)\subsetneq \Ram(B')$, hence $B'$ is also ramified at all real places of $K$ and $M'={\bf H}^3/\Gamma_{\mathcal O'}^1$ is an arithmetic hyperbolic $3$--orbifold. By hypothesis $B'$ admits embeddings of the quadratic extensions $L_1,\dots, L_r$ of $K$ and is ramified at a finite prime of $K$. By \cite[Thm 3.3]{CF}, $\mathcal O'$ admits embeddings of all of the quadratic orders $\Omega_1, \dots, \Omega_r$. It follows that $\Gamma_{\mathcal O'}^1$ contains conjugates of the loxodromic elements $\gamma_1, \dots, \gamma_r$ and that the length spectrum of the orbifold $M'$ contains $S$. To show that $M'$ is not commensurable to $M$ it suffices to show that $B\not\cong B'$, since the invariant trace field and quaternion algebra are complete commensurability class invariants \cite[Thm 8.4]{MR}. Because two quaternion algebras defined over number fields are isomorphic if and only if their ramification sets are equal, that $B\not\cong B'$ follows from the hypothesis that $\Ram(B)\subsetneq \Ram(B')$.
\end{proof}

\begin{proof}[Proof of Theorem \ref{maintheorem}]
For $M$ as in the statement of Theorem \ref{maintheorem}, let $K,B$ be the invariant trace field and quaternion algebra of $M$, and let $L_1,\dots, L_r$ be the quadratic extensions of $K$ associated to the geodesics lengths in $S$ as defined above. We may assume without loss of generality that these extensions are all distinct. That there are infinitely many non-commensurable arithmetic hyperbolic $3$--orbifolds with length spectra containing $S$ implies that there are infinitely many non-isomorphic $K$--quaternion algebras over $K$ admitting embeddings of the extensions $L_1,\dots, L_r$. This in turn implies that the degree of the compositum $L$ of $L_1,\dots, L_r$ over $K$ has degree $[L:K]=2^r$. These assertions were proven in \cite[\S 6-7]{L} in the context of arithmetic hyperbolic surfaces, and the same proofs apply in the present context of arithmetic hyperbolic $3$--orbifolds. The Galois group $\Gal(L/K)$ is isomorphic to $(\Z/2\Z)^r$ and the primes of $K$ whose Frobenius elements represent the element $(1,\dots, 1)$ correspond to those which are inert in each of the extensions $L_1/K,\dots, L_r/K$. Fix a prime $P_0$ of $K$ whose Frobenius element represents $(1,\dots, 1)$ and which does not lie in $\Ram_f(B)$. By Theorem \ref{thm:algreduction} there is a constant $C_1>0$ such that there are infinitely many $k$--tuples $P_1,\dots, P_k$ of primes of $K$, all of which are inert in the extensions $L_1/K,\dots, L_r/K$ and have norms lying within an interval of length $C_1$. We may assume that none of the primes $P_i$ ramify in $B$. As $M$ is derived from a quaternion algebra, $\pi_1(M) < \Gamma_\mathcal O^1$ for some maximal order $\mathcal O$ of $B$. Finally, by Borel \cite{B}, we have
\[\mathrm{vol}({\bf H}^3/\Gamma_\mathcal O^1)=\frac{|\Delta_K|^{3/2}\zeta_K(2)}{(4\pi^2)^{n_K-1}}\prod_{P \in\Ram_f(B)} \left(N(P)-1\right),\]
where $n_K=[K:\Q]$, $\zeta_K(s)$ is the Dedekind zeta function of $K$, and $\Delta_K$ is the discriminant of $K$.

We now use the primes $P_1,\dots, P_k$ to construct quaternion algebras $B_1,\dots, B_k$ over $K$. For each $i=1,\dots, k$, define $B_i$ to be the unique quaternion algebra over $K$ for which $\Ram(B_i)=\Ram(B)\cup \{P_0,P_i\}$. As $B$ admits embeddings of all of the quadratic extensions $L_i$, no prime of $\Ram(B)$ splits in $L_i/K$. Similarly, none of the primes $P_0, P_1,\dots, P_k$ split in $L_i/K$ for any $i$. The Albert--Brauer--Hasse--Noether theorem implies that a quaternion algebra over a number field $K$ admits an embedding of a quadratic extension of $K$ if and only if no prime which ramifies in the algebra splits in the extension of $K$. This allows us to conclude that all of the quaternion algebras which we have defined are pairwise non-isomorphic and admit embeddings of all of the $L_i$. Let $\mathcal O_1, \dots, \mathcal O_k$ be maximal orders of $B_1,\dots, B_k$. By Lemma \ref{lenspeclem}, the arithmetic hyperbolic $3$--orbifolds $M_i={\bf H}^3/\Gamma^1_{\mathcal O_i}$, which are all pairwise non-commensurable since the algebras $B_1,\dots, B_k$ are pairwise non-isomorphic, have length spectra containing $S$. By \cite{B}, the volume of $M_i$ is equal to $\mathrm{vol}({\bf H}^3/\Gamma_\mathcal O^1)\cdot (N(P_0)-1)(N(P_i)-1)$. As the $k$ primes $P_1,\dots, P_k$ have norms lying in a bounded length interval, the orbifolds $M_1,\dots, M_k$ have volumes lying in a bounded length interval. This completes the proof of Theorem \ref{maintheorem}.
\end{proof}

\section{Producing arithmetic hyperbolic $3$--manifolds}

In this section we prove a variant of Theorem \ref{maintheorem} that produces infinitely many $k$--tuples (for any $k\geq 2$) of arithmetic hyperbolic $3$--{\it manifolds} which are pairwise non-commensurable, have geodesic length spectra containing some fixed set of lengths and have volumes lying in an interval of (universally) bounded length.

\begin{cor}\label{cor:manifolds}Let $M={\bf H}^3/\Gamma_{\mathcal O}^1$ be a compact arithmetic hyperbolic $3$--manifold whose invariant quaternion algebra is ramified at some finite prime and let $S$ be a finite subset of the length spectrum of $M$. Suppose that $\pi(V,S)\to\infty$ as $V\to\infty$. Then, for every $k\geq 2$, there is a constant $C>0$ such that there are infinitely many $k$--tuples $M_1,\dots, M_k$ of arithmetic hyperbolic $3$--manifolds which are pairwise non-commensurable, have length spectra containing $S$, and volumes satisfying $\abs{\mathrm{vol}(M_i)-\mathrm{vol}(M_j)}<C$ for all $1\leq i,j\leq k$.\end{cor}
\begin{proof}
We will show that our hypotheses on $M$ imply that the orbifolds $M_1,\dots, M_k$ produced by Theorem \ref{maintheorem} in this case are all manifolds. Let $K, B$ be the invariant trace field and quaternion algebra of $M$. As $M$ is a manifold, $\Gamma_{\mathcal O}^1$ is torsion-free and so $B$ does not admit an embedding of any cyclotomic extension $F$ of $K$ with $[F:K]=2$. This follows from \cite[Thm 12.5.4]{MR} and makes use of the fact that $\Ram_f(B)$ is nonempty. The Albert--Brauer--Hasse--Noether theorem therefore implies that, for every cyclotomic extension $F$ of $K$ with $[F:K]=2$, there exists a prime $P\in \Ram(B)$ such that $P$ splits in $F/K$. Let $B_1,\dots, B_k$, $\mathcal O_1, \dots, \mathcal O_k$ and $M_1,\dots, M_k$ be as in the proof of Theorem \ref{maintheorem}. The quaternion algebras $B_1,\dots, B_k$ were defined so that $\Ram(B)\subsetneq \Ram(B_i)$, hence the Albert--Brauer--Hasse--Noether theorem again implies that no cyclotomic extension $F$ of $K$ with $[F:K]=2$ embeds into any of the quaternion algebras $B_i$. By \cite[Thm 12.5.4]{MR}, the groups $\Gamma_{\mathcal O_i}^1$ are all torsion-free, and hence the orbifolds $M_1,\dots, M_k$ are all manifolds.
\end{proof}

\section*{Acknowledgements}
The authors would like to thank Robert J. Lemke Oliver and Jesse Thorner for helpful conversations. D.M.~was supported by NSF grant DMS-1408458. P.P.~was supported by NSF grant DMS-1402268. L.T.~was supported by an AMS Simons Travel Grant and by NSF grant DMS-1440140 while in residence at the Mathematical Sciences Research Institute during the Spring 2017 semester.

\bibliographystyle{amsplain}


\end{document}